\documentclass[12pt,reqno]{amsart}
\usepackage{fullpage}
\usepackage{amssymb}
\usepackage{bbm}
\usepackage{mathrsfs}
\usepackage[cmtip,arrow]{xy}
\usepackage{pb-diagram,pb-xy}

\newcommand{\zz}{\mathbb{Z}}

\newcommand{\cc}{\mathbb{C}}

\newcommand{\qq}{\mathbb{Q}}

\newcommand{\legendre}[2]{\left(\frac{#1}{#2}\right)}

\newcommand{\onetoone}{\hookrightarrow}

\newcommand{\isom}{\cong}

\newcommand{\GL}{\mathrm{GL}}

\newcommand{\Gal}{\mathrm{Gal}}
\newcommand{\N}{\mathbf{N}}
\newcommand{\p}{\mathfrak{p}}

\def\Re{\mathrm{Re}}

\newtheorem{prop}{Proposition}
\newtheorem{thm}{Theorem}
\newtheorem{lma}{Lemma}

\theoremstyle{remark}
\newtheorem*{rmk}{Remark}

\usepackage{fullpage}
\begin{document}
\title{A variant of the Barban-Davenport-Halberstam Theorem\\
%(Appeared in \textit{International Journal of Number Theory})
}

\author{Ethan Smith}

\address{
Centre de recherches math\'ematiques\\
Universit\'e de Montr\'eal\\
P.O. Box 6128\\
Centre-ville Station\\
Montr\'eal, Qu\'ebec\\
H3C 3J7\\
Canada;
\and
Department of Mathematical Sciences\\
Michigan Technological University\\
1400 Townsend Drive\\
Houghton, Michigan\\
49931-1295\\
USA
}
\email{ethans@mtu.edu}
\urladdr{www.math.mtu.edu/~ethans}

\begin{abstract}
Let $L/K$ be a Galois extension of number fields.  The problem of counting 
the number of prime ideals $\p$ of $K$ with fixed Frobenius class in 
$\Gal(L/K)$ and norm satisfying a congruence condition is considered.
We show that the square of the error term arising from the Chebotar\"ev 
Density Theorem for this problem is small ``on average."  The result may 
be viewed as a variation on the classical Barban-Davenport-Halberstam
Theorem.
\end{abstract}

\maketitle

\section{Introduction}

The Barban-Davenport-Halberstam Theorem~\cite{Bar:1964,DH:1966,DH:1968}
concerns the average square error for the Prime Number Theorem 
for primes in arithmetic progressions.  In particular, if we let 
\begin{equation*}
\theta(x;q,a):=\sum_{\substack{p\le x\\ p\equiv a\pmod q}}\log p,
\end{equation*}
then the Barban-Davenport-Halberstam Theorem states that,
for every fixed $M>0$,
\begin{equation*}
\sum_{q\le Q}\sum_{\substack{a=1\\ (a,q)=1}}^q\left(\theta(x;q,a)
-\frac{x}{\varphi(q)}\right)^2
\ll xQ\log x,
\end{equation*}
provided that $x(\log x)^{-M}\le Q\le x$.
See~\cite[p.~169]{Dav:1980} for a statement of this theorem with 
$\theta(x;q,a)$ replaced by the Chebychev function $\psi(x;q,a)$.

If one adopts the point of view that the Prime Number Theorem for primes in 
arithmetic progressions is a specialization of the Chebotar\"ev Density 
Theorem~\cite[p.~143]{IK:2004} for the cyclotomic extensions $\qq(\zeta_q)/\qq$, 
then, in light of the Kronecker-Weber Theorem~\cite[p.~210]{Lan:1994}, 
the Barban-Davenport-Halberstam Theorem 
may be interpreted as giving a bound for the average square error for the Chebotar\"ev Density 
Theorem taken over all Abelian extensions of $\qq$.
It would be natural then to study averages that include non-Abelian extensions
as well.  In this paper, we consider averages where the Galois group varies 
over groups of the form $G_1\times G_2$, where $G_1$ is a fixed (potentially non-Abelian) 
group and $G_2\subseteq(\zz/q\zz)^*$ with $q$ varying.

This work is carried out for the purpose of providing a tool to study the 
``average Lang-Trotter" problem in number fields.
Much is owed to the related paper of Murty and Murty~\cite{MM:1987}.

\section{Statement of results}

Let $L/K$ be a Galois extension of number fields with group $G$, and 
let $C$ be a fixed conjugacy class in $G$.
For any pair of integers $q$ and $a$, we define 
\begin{equation*}
\theta(x;C,q,a):=
\sum_{\substack{\N\p\le x\\ \legendre{L/K}{\p}=C\\ 
\N\p\equiv a\pmod q}}
\log\N\p,
\end{equation*}
where the sum is taken over all finite primes $\p$ of $K$ 
that do not ramify in $L$, $\N=\N_{K/\qq}$ denotes the norm, and 
$\legendre{L/K}{\p}$ denotes the Frobenius class of $\p$ in $G$.

For each positive integer $q$, 
let $\zeta_q$ denote a primitive $q$-th root of unity. 
If $L\cap K(\zeta_q)=K$, then 
\begin{equation*}
\mathcal G_{q}:=\Gal(L(\zeta_q)/K)\isom\Gal(K(\zeta_q)/K)\times G.
\end{equation*}
See~\cite[p.~267]{Lan:2002}.
As in~\cite{Smi:2009}, there is a natural composition of maps
\begin{equation}\label{image mod q}
\begin{diagram}
\node{\Gal(K(\zeta_q)/K)}\arrow{e,J}
\node{\Gal(\qq(\zeta_q)/\qq)}\arrow{e,t}{\sim}
\node{(\zz/q\zz)^*.}
\end{diagram}
\end{equation}
Letting $G_{K,q}$ denote the image of this map and $\varphi_K(q):=|G_{K,q}|$, the 
Chebotar\"ev Density Theorem gives the asymptotic identity
\begin{equation}\label{Cheb}
\theta(x;C,q,a)\sim\frac{1}{\varphi_K(q)}\frac{|C|}{|G|}x
\end{equation}
as $x\rightarrow\infty$, provided that $a\in G_{K,q}$ and $L\cap K(\zeta_q)=K$. 
We remark that the Frobenius class of $\p$ in $\Gal(K(\zeta_q)/K)$ is 
determined by the residue of $\N\p$ modulo $q$.

In this paper, we establish a variant of the classical 
Barban-Davenport-Halberstam Theorem similar to the variant of the classical 
Bombieri-Vinogradov Theorem proved by Murty and Murty in~\cite{MM:1987}.
That is, we show an upper bound for the average square error 
in the asymptotic~\eqref{Cheb} when averaging over the set of all $q$ such 
that $L\cap K(\zeta_q)=K$ and over all $a\in G_{K,q}$.
More precisely, we prove the following result.

\begin{thm}\label{main thm}
Let $M>0$.
If $x(\log x)^{-M}\le Q\le x$, then
\begin{equation}
\sum_{q\le Q}\strut^\prime\sum_{a\in G_{K,q}}\left(\theta(x;C,q,a)
-\frac{1}{\varphi_K(q)}\frac{|C|}{|G|}x\right)^2
\ll xQ\log x,
\end{equation}
where the prime on the outer summation indicates that 
the sum is to be restricted to those $q\le Q$ satisfying $L\cap K(\zeta_q)=K$.
The constant implied by the symbol $\ll$ depends on $L$ and $M$.
\end{thm}
\begin{rmk}
By the triangle inequality, Theorem~\ref{main thm} holds in the case that $C$ is a union of 
conjugacy classes in $G=\Gal(L/K)$ as well.
\end{rmk}
\begin{rmk}
The main result of~\cite{MM:1987} is precisely the Bombieri-Vinogradov type
analogue of Theorem~\ref{main thm} in the case $K=\qq$.
M. R. Murty and K. Petersen have recently extended the main result 
of~\cite{MM:1987} to the setting of number fields.  
However, their average on $q$ is restricted to $q$ satisfying 
$L\cap\qq(\zeta_q)=\qq$.
The author is  grateful for receiving a copy of their preprint~\cite{MP-pp}.
\end{rmk}

As mentioned in the introduction,
this work is carried out for the purpose of providing a tool to study the 
average Lang-Trotter problem in number fields.  In particular, it 
useful for extending the works~\cite{DP:2004, FJKP} to the 
``totally non-Abelian" setting.  As another application, we give the following.
\begin{thm}\label{main thm2}
Suppose that $K=\qq$ and that $L/\qq$ is a totally non-Abelian Galois extension
(i.e., $L\cap\qq(\zeta_q)=\qq$ for all $q\ge 1$).
There exist constants $c,c'$ so that, for each fixed $M>0$,
\begin{equation}
\sum_{q\le x}\sum_{\substack{a=1\\ (a,q)=1}}^q\left(\theta(x;C,q,a)
-\frac{|C|}{\varphi(q)|G|}x\right)^2=
\frac{|C|}{|G|}x^2\log x+cx^2+O\left(\frac{x^2}{(\log x)^M}\right);
\end{equation}
and if $1\le Q\le x$,
\begin{equation}
\begin{split}
\sum_{q\le Q}\sum_{\substack{a=1\\ (a,q)=1}}^q\left(\theta(x;C,q,a)
-\frac{|C|}{\varphi(q)|G|}x\right)^2&=\frac{|C|}{|G|}xQ\log x
-\frac{|C|^2}{|G|^2}xQ\log(x/Q)+c'xQ\\
&\quad+O\left(Q^\frac{5}{4}x^\frac{3}{4}\right)
+O\left(\frac{x^2}{(\log x)^{M}}\right).
\end{split}
\end{equation}
\end{thm}

\begin{rmk}
See~\cite{Smi:2010} for a version of this result when $K=L$ and $K$ is not 
necessarily equal to $\qq$.
\end{rmk}

We omit the proof of Theorem~\ref{main thm2} since it follows from 
Theorem~\ref{main thm} by adapting 
the techniques of Hooley~\cite[pp. 209-212]{Hoo:1975} in the obvious manner.
It would be natural to try to remove the restrictions on $K$ and $L$ from the 
statement of Theorem~\ref{main thm2}.  However, if $K\ne\qq$, then one seems 
to be left with the task of counting the number of prime ideals of a given norm and 
fixed Frobenius class.
% - a version of Chebotar\"ev Density Theorem for intervals of 
%length zero as it were.  
In the case, that $L$ is not linearly disjoint from every 
cyclotomic extension, Theorem~\ref{main thm} requires that we remove 
the ``offending moduli" from the outer sum over $q$.  The result is that 
Hooley's trick (see~\cite[p. 210]{Hoo:1975}) of ``switching divisors" does not seem to work
anymore.

The remainder of the article is outlined as follows.  In Section~\ref{artin l-fcts}, we recall 
some useful facts about characters and $L$-functions.  
Section~\ref{sw sect} is devoted to the 
proof of Siegel-Walfisz type results.
The final section is concerned with the proof of 
Theorem~\ref{main thm}.

\section{Class functions, characters, and Artin $L$-functions}
\label{artin l-fcts}

In this section, we set up some general notation and 
recall some of the necessary background information concerning Galois representations 
and Artin $L$-functions.
For a number field $F$, we will write $d_F$ for the discriminant of $F$.
If $L/F$ is an extension of number fields, we will write 
$n_{L/F}:=[L:F]$ for the degree of the extension.  In the case that the base field is 
$\qq$, we will simply write $n_L:=[L:\qq]$.
%We also write $\mathfrak{D}(L/F)$ and $\mathfrak{d}(L/F)$ for the different and 
%relative discriminant, respectively.  See, for example,~\cite[Chap. III]{Lan:1994} 
%for definitions and basic properities.

Suppose that $L/F$ is a Galois extension of number fields, say with group $G$.  
Consider a group representation $\rho: G\rightarrow\GL_n(\cc)$.  The character 
associated to $\rho$ is the function $\eta=\eta_\rho$ defined by 
$\eta(\sigma):=\mathrm{Tr}(\rho(\sigma))$.  
For each finite prime $\p$ of $F$ unramified in $L$, 
choose a prime $\mathfrak P|\p$ in $L$, and let $\sigma_{\mathfrak P}$ denote the Frobenius
at $\mathfrak P$.  Further, put
\begin{equation}
L_\p(s,\eta):=
\det\left(I-\rho(\sigma_\mathfrak P)\N\p^{-s}\right)^{-1},
\end{equation} 
and note that this definition does not depend on the choice of $\mathfrak{P}|\p$.
For the ramified primes, there is a similar but slightly more complicated definition for 
$L_\p(s,\eta)$, and we refer the reader to~\cite[pp.~8-9]{Mar:1975} or~\cite[p.~27]{MM:1997} for the details.
For $\Re(s)>1$, the Artin $L$-series associated to $\rho$ or $\eta$ is 
defined by the Euler product
\begin{equation}
L(s,\eta):=\prod_\p L_\p(s,\eta),
\end{equation}
where the product is over all the prime ideals $\p$ of $F$ (including the ramified primes).

Associated to this $L$-series is an ideal of  $F$ called the \textit{Artin conductor}.
We denote this ideal by $\mathfrak{f}(\eta, L/F)$.  It is defined 
in terms of higher ramification groups, and we refer the reader to~\cite[pp.~13-14]{Mar:1975}
or~\cite[p.~28]{MM:1997} for more detail.
The \textit{conductor} of the $L$-series is then defined by
\begin{equation*}
A(\eta):=|d_F|^{\eta(1)}\N_{F/\qq}(\mathfrak{f}(\eta,L/F)).
\end{equation*}

A function $\xi: G\rightarrow\cc$ satisfying $\xi(g^{-1}\sigma g)=\xi(\sigma)$ for all $g\in G$ 
is said to be a class function on $G$.
Note that the character of a group representation is a class function.
Given any class function $\xi$ on $G$, we define the weighted prime counting function
\begin{equation*}
\theta(x;L/F,\xi):=\sum_{\N\p\le x}\xi(\sigma_\p)\log\N\p,
\end{equation*}
where the sum is over the prime ideals of $F$ that do not ramify in $L$, and 
$\sigma_\p$ denotes any element of the Frobenius class at $\p$.
We also define the function
\begin{equation*}
\psi(x;L/F,\xi):=
\sum_{\N\p^m\le x}\xi(\sigma_\p^m)\log\N\p,
\end{equation*}
where now we include all powers of prime ideals of $F$ that do not ramify in $L$.
For a prime $\p$ of $F$ ramified in $L$ and $\mathfrak{P}$ a prime of $L$ lying above $\p$, 
let $D_\mathfrak{P}$ and $I_\mathfrak{P}$ denote the decomposition and inertia group at $\mathfrak{P}$ respectively.
We extend the definition of $\xi$ to ramfied $\p$ by
\begin{equation*}
\xi(\sigma_\p^m):=\frac{1}{|I_\mathfrak{P}|}\sum_g\xi(g),
\end{equation*}
where the sum is over all $g\in D_\mathfrak{P}$ such that $g\equiv\sigma_\mathfrak{P}^m\pmod{I_\mathfrak{P}}$.
We now define
\begin{equation*}
\widetilde{\psi}(x;L/F,\xi):=
\sum_{\N\p^m\le x}\xi(\sigma_\p^m)\log\N\p,
\end{equation*}
where the sum is over all powers of prime ideals $\p^m$ of $F$ including the ramified ones.
We will sometimes suppress the field extension in the notation for these functions 
when there is no danger of confusion.

It will sometimes be necessary to switch between $\theta$ and $\widetilde\psi$.
The difference between these functions is
\begin{equation*}
\theta(x;\xi)-
\widetilde{\psi}(x;\xi)
=\sum_{\substack{\N\p^m\le x,\\ m\ge 2}}\xi(\sigma_\p^m)\log\N\p
+\sum_{\substack{\N\p\le x\\ \p\text{ ram.}}}\xi(\sigma_\p)\log\N\p.
\end{equation*}
The first sum on the right is over all prime ideals $\p$ of $F$ and all integers $m\ge 2$, 
and is trivially bounded by $||\xi||n_F\sqrt x\log x$, where 
$||\xi||:=\max_{g\in G}\xi(g)$.
The second sum on the right is over only those prime ideals $\p$ of $F$ which ramify in $L$.
For the second sum, we have the bound
\begin{equation*}
\left|\sum_{\substack{\N\p\le x\\ \p\text{ ram.}}}\xi(\sigma_\p)\log\N\p\right|
\le||\xi||\sum_{\p\mid\mathfrak{d}(L/F)}\log\N\p
\le\frac{2||\xi||}{n_{L/F}}\log|d_{L}|
\end{equation*}
by~\cite[Prop. 5]{Ser:1981}.
Therefore,
\begin{equation}\label{swap theta and psi}
\left|\widetilde{\psi}(x;\xi)-
\theta(x;\xi)\right|
\le ||\xi|| n_F\left(
\sqrt x\log x+\frac{2}{n_L}\log|d_L|
\right).
\end{equation}

Any class function $\xi$ can be written as
\begin{equation*}
\xi=\sum_\eta a_\eta\eta,
\end{equation*}
where $a_\eta\in\cc$ and $\eta$ ranges over the irreducible characters of $G$.  
In particular, if $C$ is a conjugacy class of $G$ and $g_C\in C$, then the characteristic function of 
$C$ may be decomposed as
\begin{equation*}
\delta_C=\frac{|C|}{|G|}\sum_{\eta}\overline\eta(g_C)\eta .
\end{equation*}
In general, we have the identity
\begin{equation*}
\widetilde\psi(x;L/F,\xi)=
-\frac{1}{2\pi i}\sum_\eta a_\eta\int_{(2)}\frac{L'(s,\eta)}{L(s,\eta)}\frac{x^s}{s}ds,
\end{equation*}
where the line of integration is $\mathrm{Re}(s)=2$.

Now, suppose that $H$ is a subgroup of $G=\Gal(L/F)$ and $E=L^H$ is the fixed field of $H$.
If $\xi$ is a class function on $H$, we may define the induced class function 
$\mathrm{Ind}_H^G(\xi)$ as follows.  First extend $\xi$ to all of $G$ by 
$\xi(\sigma)=0$ for all $\sigma\in G\backslash H$; then set
\begin{equation*}
\mathrm{Ind}_H^G(\xi)(\sigma):=\frac{1}{|H|}\sum_{g\in G}\xi(g^{-1}\sigma g).
\end{equation*}
An important feature of Artin $L$-functions is their invariance under 
induction.  That is, if $\eta$ is a character of $H$, then 
\begin{equation*}
L(s, L/E, \eta)=L(s, L/F, \mathrm{Ind}_H^G(\eta)).
\end{equation*}
See~\cite[p.~9]{Mar:1975} for example.
Thus, it follows that if $\xi$ is a class function on $H$, then
\begin{equation*}
\widetilde\psi(x;L/E,\xi)=\widetilde\psi(x;L/F,\mathrm{Ind}_H^G(\xi)).
\end{equation*}

Now suppose that $G=G_1\times G_2$, and let $\rho_1, \rho_2$ define irreducible representations of 
$G_1$ and $G_2$ respectively.  Then the tensor product representation, defined 
by
\begin{equation*}
(\rho_1\otimes\rho_2)(g_1,g_2):=\rho_1(g_1)\otimes\rho_2(g_2),
\end{equation*}
is an irreducible representation of $G$.
Moreover, every irreducible representation of $G$ arises in this way.
If $\eta_1,\eta_2$ are the characters associated to $\rho_1,\rho_2$ respectively, then
the character of the tensor product representation is given by
\begin{equation*}
(\eta_1\otimes\eta_2)(g_1,g_2)=\eta_1(g_1)\eta_2(g_2).
\end{equation*}
See~\cite[p.~26-28]{Ser:1977}.
Now, suppose $L_1/F$ and $L_2/F$ are Galois extensions of number fields with 
$\Gal(L_1/F)=G_1$ and $\Gal(L_2/F)=G_2$.  Suppose further that 
$L_1\cap L_2=F$, and set $L=L_1L_2$.  Then $L/F$ is Galois with group 
$G_1\times G_2$~\cite[p.~267]{Lan:2002}.
Finally, we note that the Artin conductors satisfy the relationship
\begin{equation}\label{artin cond of tensor}
\mathfrak{f}(\eta_1\otimes\eta_2, L/F)
\mid\mathfrak{f}(\eta_1,L_1/F)^{\eta_2(1)}\mathfrak{f}(\eta_2,L_2/F)^{\eta_1(1)}.
\end{equation}
See~\cite[p.~80]{Mar:1975} for example.

\section{Siegel-Walfisz type estimates}\label{sw sect}

In this section we assume that $L/E$ is an Abelian extension of number fields.
It follows that $L(\zeta_q)/E$ is also an Abelian extension.
In fact, $\Gal(L(\zeta_q)/E)\onetoone\Gal(L/E)\times (\zz/q\zz)^*$.
See~\cite[p. 267]{Lan:2002} for example.
By the discussion in the last paragraph 
of Section~\ref{artin l-fcts}, any irreducible character $\eta$ of $\Gal(L(\zeta_q)/E)$
arises as $\eta=\omega\otimes\chi$, where $\omega$ is 
an irreducible character of $\Gal(L/E)$ and $\chi$ is a 
Dirichlet character modulo $q$.
By~\eqref{artin cond of tensor}, we have the following conductor relation
\begin{equation}\label{bound conductors}
\begin{split}
A(\omega\otimes\chi,L(\zeta_q)/E)
%&=d_E\N(\mathfrak{f}(\omega\otimes\chi,L(\zeta_q)/E)\\
&\le |d_E|\N_{E/\qq}(\mathfrak{f}(\omega,L/E))\N_{E/\qq}(\mathfrak f(\chi,E(\zeta_q)/E))\\
&\le A(\omega,L/E)q^{n_E}.
\end{split}
\end{equation}
The first objective of this section is to prove the following proposition, 
which is a generalization of Lemma 6.2 of~\cite[p.~267]{MM:1987}.

\begin{prop}\label{SW gen}
Let $L/E$ be an Abelian extension of number fields, $\omega$ an  irreducible character of 
$\Gal(L/E)$, and let $1<q,N$ be any integers so that
 $A(\omega)q^{n_E}\le(\log x)^N$.
Then there exists a positive constant $c_0(E,N)$ so that,
for any Dirichlet character $\chi$ modulo $q$ such that 
$\omega\otimes\chi$ is not trivial,
\begin{equation*}
\widetilde\psi(x;L(\zeta_q)/E,\omega\otimes\chi)
\ll
(\log A(\omega\otimes\chi))
x\exp\left\{
-c_0(E,N)\sqrt{\log x}
\right\}.
\end{equation*}  
\end{prop}

The proof of this proposition relies heavily on the following,
the content of which is drawn from Propositions 3.4 and 3.8 
of~\cite[pp. 255, 258]{MM:1987}.
\begin{lma}\label{zero free region}
Let $F/E$ be an Abelian extension of number fields, let 
$\eta$ be a character of $\Gal(F/E)$, and 
let $\mathcal{L}(t):=\frac{1}{2}\log A(\eta)+n_E\log(|t|+2)$.
There is an absolute positive constant $c_1$ such that 
$L(s,\eta)$ has at most one zero $\sigma+it$ in 
the region
\begin{equation*}
1-\frac{c_1}{\mathcal{L}(t)}\le\sigma\le 1.
\end{equation*}
If such a zero exists, then it is real and simple, and $\eta$ must be a 
character of order dividing $2$.
Furthermore, we refer to such a zero as an exceptional zero of 
$L(s,\eta)$ and denote it by $\beta_\eta$.

Let $\epsilon>0$, and let $\tilde E$ denote the normal closure of $E$ over 
$\qq$.  If $\eta$ is a character of $\Gal(F/E)$ for which $\beta_\eta$ 
exists, then there is a positive constant $c_2(\epsilon)$ so that
\begin{equation*}
\beta_\eta\le
\max\left\{
1-\frac{c_2(\epsilon)}{(|d_{\tilde E}|^2A(\eta))^\epsilon},
1-\frac{1}{16n_{\tilde E}\log(|d_{\tilde E}|^2A(\eta))}
\right\}.
\end{equation*}
\end{lma}

\begin{rmk}
We observe that if $\eta$ is a character for which $\beta_\eta$ exists and 
$\beta_\eta<1/2$, then the functional equation for $L(s,\eta)$ implies 
that $1-\beta_\eta$ is also a zero since $\eta=\overline\eta$.  
In this case, we would have
\begin{equation*}
1-\frac{c_1}{\mathcal L(0)}\le \beta_\eta<\frac{1}{2}<1-\beta_\eta,
\end{equation*}
which contradicts Lemma~\ref{zero free region}.
Thus, we may assume that if $\beta_\eta$ exists, then
\begin{equation*}
\frac{1}{2}\le 1-\frac{c_1}{\mathcal L(0)}\le \beta_\eta.
\end{equation*}
%In particular, we have
%\begin{equation}\label{bound c}
%c\le \frac{1}{2}\mathcal L(0)\le\frac{1}{2}\mathcal L(t)
%\end{equation}
%for all $t$.
\end{rmk}

\begin{rmk}
Applying Lemma~\ref{zero free region} to the character $\overline\eta$ and appealing 
to the functional equation for $L(s,\eta)$, we see that $L(s,\eta)$ has at most one 
zero $\sigma+it$ in the region
\begin{equation*}
0\le\sigma\le\frac{c_1}{\mathcal{L}(t)}.
\end{equation*}
Furthermore, if such a zero exists, it must be simple, $\eta=\overline\eta$ must be a 
character of order dividing $2$, and we deduce that this zero equal to $1-\beta_\eta$.
\end{rmk}

\begin{rmk}
The second part of Lemma~\ref{zero free region} depends on 
Siegel's Theorem~\cite[p. 126]{Dav:1980},
and as such is ineffective.  This, of course, means that 
Proposition~\ref{SW gen} is ineffective as well.
\end{rmk}

\begin{proof}[Proof of Proposition~\ref{SW gen}]
The proof follows the traditional method as in~\cite{LO:1977}.
For convenience, we write $\eta=\omega\otimes\chi$.
We begin with integral identity
\begin{equation*}
\widetilde\psi(x;L(\zeta_q)/E,\eta)=
-\frac{1}{2\pi i}\int_{(2)}\frac{L'(s,\eta)}{L(s,\eta)}\frac{x^s}{s}ds.
\end{equation*}
With $\sigma_0=1+(\log x)^{-1}$ and $T$ a real parameter, we 
estimate the right-hand side by the truncated integral
\begin{equation*}
I_\eta(x,T):=
-\frac{1}{2\pi i}\int_{\sigma_0-iT}^{\sigma_0+iT}\frac{L'(s,\eta)}{L(s,\eta)}\frac{x^s}{s}ds.
\end{equation*}
We note that our definition of $I_\eta(x,T)$ differs from that 
in~\cite[p. 440]{LO:1977} by a negative sign.

Arguing as in~\cite[pp. 424-428]{LO:1977}, for $x\ge 2$ and $T>0$, we have
\begin{equation}\label{approx by truncated integral}
\widetilde\psi(x;L(\zeta_q)/E,\eta)
-I_\eta(x,T)\ll n_E\log x+\frac{n_Ex(\log x)^2}{T},
\end{equation}
where the implied constant is absolute.
Since $\chi$ is assumed to be nontrivial, it follows that 
$\eta=\omega\otimes\chi$ is as well.  
Now, provided that $x,T\ge 2$ and that $T$ does not coincide with the imaginary 
part of any zero of $L(s;L(\zeta_q)/E,\eta)$,
in~\cite[p.450]{LO:1977}, we find that
\begin{equation}\label{approx by sums over zeros}
I_\eta(x,T)
+\sum_{\substack{\rho=\beta+i\gamma\\ |\gamma|<T}}\frac{x^\rho}{\rho}
-\sum_{\substack{\rho=\beta+i\gamma\\ |\rho|<\frac{1}{2}}}\frac{1}{\rho}
\ll
\log A(\eta)+n_E\log x+\frac{x\log x}{T}\left(
\log A(\eta)+n_E\log T\right),
\end{equation}
where both sums are taken over the nontrivial zeros $\rho$ of
$L(s;L(\zeta_q)/E,\eta)$ and the implied constant is absolute.
If, however, $T$ does coincide with the imaginary part of some zero $\rho$, 
we may argue as in~\cite[p.451]{LO:1977} and correct the problem 
(by simply increasing the constant implied by the $\ll$ notation).

By Lemma~\ref{zero free region}, if $\rho=\beta+i\gamma\ne\beta_\eta$ 
is not an exceptional zero of 
$L(s;L(\zeta_q)/E,\eta)$ and $|\gamma|<T$, then
\begin{equation*}
|x^\rho|\le 
x\exp\left\{-\frac{c_1\log x}{\mathcal L(T)}\right\}.
\end{equation*}
Let $N_\eta(t)$ denote the number of nontrivial zeros $\rho=\beta+i\gamma$ 
of $L(s;L(\zeta_q)/E,\eta)$ satisfying the condition $|\gamma-t|\le 1$.
By Lemma 5.4 of~\cite[p. 436]{LO:1977}, 
$N_\eta(t)\ll\log A(\eta)+n_E\log(|t|+ 2)$.
Thus,
\begin{equation*}
\begin{split}
\sum_{\substack{\rho=\beta+i\gamma\\ \rho\neq\beta_\eta\\ |\rho|\ge 1/2\\ |\gamma|<T}}
\left|\frac{x^\rho}{\rho}\right|
&\ll
x\exp\left\{-\frac{c_1\log x}{\mathcal L(T)}\right\}
\sum_{j<T}\frac{N_\eta(j)}{j}\\
&\ll
x\exp\left\{-\frac{c_1\log x}{\mathcal L(T)}\right\}
\left(\log A(\eta)+n_E\log T\right)
\log T,
\end{split}
\end{equation*}
the implied constants being absolute.
By the second remark following Lemma~\ref{zero free region}, if 
$\rho=\beta+i\gamma\ne 1-\beta_\eta$, then 
$\beta\ge \frac{c_1}{\mathcal L(t)}$ and hence
\begin{equation*}
\frac{1}{|\rho|}\le\frac{\mathcal L(\gamma)}{c_1}
\ll\log A(\eta)+n_E,
\end{equation*}
the implied constant being absolute.
Thus, 
\begin{equation*}
\begin{split}
\sum_{\substack{\rho=\beta+i\gamma\\ \rho\ne 1-\beta_\eta\\ |\rho|<\frac{1}{2}}}
\left|\frac{x^\rho}{\rho}\right|
+\left|\frac{1}{\rho}\right|
&\ll x^{1/2}\sum_{\substack{|\rho|<\frac{1}{2}\\ \rho\ne 1-\beta_\eta}}\frac{1}{|\rho|}
\ll x^{1/2}N_\eta(0)\max_{\substack{|\rho|<\frac{1}{2}\\ \rho\ne 1-\beta_\eta}}
\left\{\frac{1}{|\rho|}\right\}\\
&\ll x^{1/2}\left(\log A(\eta)+n_E\right)^2,
\end{split}
\end{equation*}
where again, all implied constants are absolute.
By the Mean Value Theorem,
\begin{equation*}
\frac{x^{1-\beta_\eta}}{1-\beta_\eta}-\frac{1}{1-\beta_\eta}
=\frac{x^{1-\beta_\eta}-1}{1-\beta_\eta}=x^\sigma\log x
\end{equation*}
for some $0<\sigma<1-\beta_\eta\le\frac{1}{2}$.
By the second part of Lemma~\ref{zero free region}, 
there exists a positive constant $c_3(E,\epsilon)$ depending on 
$E$ and $\epsilon$ so that
\begin{equation*}
\frac{1}{2}\le\beta_\eta\le 1-\frac{c_3(E,\epsilon)}{A(\eta)^\epsilon}.
\end{equation*}
Hence, 
\begin{equation*}
\begin{split}
\left|\frac{x^{\beta_\eta}}{\beta_\eta}\right|
\le 2x^{\beta_\eta}
\le 2x\exp\left\{
-c_3(E,\epsilon)\frac{\log x}{A(\eta)^\epsilon}
\right\}
\ll x\exp\left\{
-c_3(E,\epsilon)(\log x)^{1-N\epsilon}
\right\},
\end{split}
\end{equation*}
where the implied constant is absolute.
Therefore,
\begin{equation}\label{bound sums over zeros}
\begin{split}
\sum_{\substack{\rho=\beta+i\gamma\\ |\gamma|<T}}\frac{x^\rho}{\rho}
-\sum_{\substack{\rho=\beta+i\gamma\\ |\rho|<\frac{1}{2}}}\frac{1}{\rho}
&\ll
x\exp\left\{
-c_1\frac{\log x}{\mathcal L(T)}
\right\}
\left(\log A(\eta)+\log T\right)
\log T
+\sqrt x(\log A(\eta)+n_E)^2\\
&\quad+x\exp\left\{
-c_3(E,\epsilon)(\log x)^{1-N\epsilon}
\right\}
+\sqrt x(\log x),
\end{split}
\end{equation}
where the implied constant is absolute.
Choosing $\epsilon=1/2N$,
\begin{equation*}
T=\exp\left\{
\frac{1}{n_E}\left(\sqrt{\log x}-\frac{1}{2}\log A(\eta)\right)
\right\},
\end{equation*}
and combining
~\eqref{approx by truncated integral},~\eqref{approx by sums over zeros},
and~\eqref{bound sums over zeros}, the proposition follows.
\end{proof}

%\section{An error bound for the Chebotar\"ev Density Theorem}\label{ECDT}

We now return to the situation described in the introduction.  That is, we let 
$L/K$ be a Galois (not necessarily Abelian) extension of number fields with group $G$.  
Further, we let $C$ be a fixed conjugacy class in $G$, 
and we let $\delta_C$ denote the characteristic function of $C$.
Given a Dirichlet character $\chi$ modulo $q$, we now wish to approximate 
\begin{equation*}
\theta(x;\delta_C\otimes\chi)=\sum_{\substack{\N\p\le x\\ \legendre{L/K}{\p}=C}}\chi(\N\p)\log\N\p.
\end{equation*}

%In this section, we apply Proposition~\ref{SW gen} and the Effective 
%Chebotar\"ev Theorem~\cite{LO:1977} to write a bound for the error 
%in~\eqref{Cheb}.  The result, however, will be ineffective due to 
%the nature of Proposition~\ref{SW gen}.

%Via~\eqref{image mod q}, we see that there is a one to one correspondence 
%between 
% conjugacy classes of $\Gal(K(\zeta_q)/K)$ and congruences classes 
% $a\in G_{K,q}\subseteq(\zz/q\zz)^*$.  As remarked in the introduction, 
% for a prime $\p$ of $K$,  the Frobenius class of $\p$ in $\Gal(K(\zeta_q)/K)$
% is determined solely by the value $\N\p$ modulo $q$.  
% Thus, we let $\delta_{q,a}$ 
% denote the characteristic function picking out those powers of prime ideals $\p^m$ with 
% $\N\p^m\equiv a\pmod q$.
%If $\eta$ is a character associated to an irreducible representation of 
%$\Gal(K(\zeta_q)/K)\isom G_{K,q}$, 
%it follows that $\eta(\sigma_\p)=\chi(\N\p)$ for some Dirichlet character modulo $q$.

%For each pair $q$ and $a$, we have the usual decomposition
%\begin{equation}\label{usual decomp}
%\theta(x;C,q,a)
%%=\theta(x;L(\zeta_q)/K,\delta_C\otimes\delta_{q,a})
%=\frac{1}{\varphi(q)}
%\sum_{\chi\bmod q}\overline\chi(a)\theta(x;\delta_C\otimes\chi).
%\end{equation}
%Of course, there is a bit of ``redundancy" in this expression. 
%Then
%\begin{equation*}
%\theta(x;\delta_C\otimes\chi_1)
%=
%\theta(x;\delta_C\otimes\chi_2)
%\end{equation*}
%whenever $\chi_1\equiv\chi_2\pmod{G_{K,q}^\perp}$.

In the proofs, it will be more convenient to work with $\widetilde\psi$ instead of $\theta$.
By~\eqref{swap theta and psi}, we have
\begin{equation*}
\left|
\widetilde\psi(x;L(\zeta_q)/K,\delta_C\otimes\chi)
-\theta(x;\delta_C\otimes\chi)
\right|
\le
n_K\left(
\sqrt x\log x+\frac{2}{\varphi_K(q)n_L}\log|d_{L(\zeta_q)}|
\right)
\end{equation*}
as $n_{L(\zeta_q)}=\varphi_K(q)n_L$.
Since $[L(\zeta_q):\qq(\zeta_q)]=\varphi_K(q)n_L/\varphi(q)$, 
by Lemma 7 of~\cite[p. 143]{Sta:1974}, we have
$|d_{L(\zeta_q)}|\le |d_L|^{\varphi_K(q)}\cdot|d_{\qq(\zeta_q)}|^{n_L\varphi_K(q)/\varphi(q)}$.
%For an upper bound on $|d_{L(\zeta_q)}|$, we have the following lemma.
%\begin{lma}
%$|d_{L(\zeta_q)}|\le |d_L|^{\varphi_K(q)}\cdot |d_{\qq(\zeta_q)}|^{n_{L}\varphi_K(q)/\varphi(q)}$.
%\end{lma}
%\begin{proof}
%By~\cite[Lemma 6]{Sta:1974}, 
%$\mathfrak{D}(L(\zeta_q)/L)\mathfrak a=\mathfrak{D}(\qq(\zeta_q)/\qq)$ 
%for some nonzero integral ideal $\mathfrak a$ of the ring of integers of $L(\zeta_q)$.
%Multiplying both sides of this equation by 
%$d_L^{[L(\zeta_q):L]}=d_L^{\varphi_L(q)}=d_L^{\varphi_K(q)}$ 
%and taking norms all the way down to $\qq$, we have
%\begin{equation*}
%d_L^{[L(\zeta_q):L]}\N_{L(\zeta_q)/\qq}\left(\mathfrak{D}(L(\zeta_q)/L)\mathfrak a\right)
%=d_L^{\varphi_K(q)}\N_{L(\zeta_q)/\qq}\left(\mathfrak{D}(\qq(\zeta_q)/\qq)\right).
%\end{equation*}
%The left-hand side of this equation is equal to 
%$d_{L(\zeta_q)}\N_{L(\zeta_q)/\qq}(\mathfrak a)$, while 
%the right is equal to 
%\begin{equation*}
%d_L^{\varphi_K(q)}\cdot d_{\qq(\zeta_q)}^{[L(\zeta_q):\qq(\zeta_q)]}
%=d_L^{\varphi_K(q)}\cdot d_{\qq(\zeta_q)}^{n_L\varphi_K(q)/\varphi(q)}.
%\end{equation*}
%This completes the proof of the lemma.
%\end{proof}
In~\cite[p. 12]{Was:1997}, we find the identity
\begin{equation*}
d_{\qq(\zeta_q)}=(-1)^{\varphi(q)/2}\frac{q^{\varphi(q)}}{\prod_{\ell |q}\ell^{\varphi(q)/(\ell-1)}},
\end{equation*} 
where the product is over the distinct primes $\ell$ dividing $q$.
Thus, we have
\begin{equation*}
\frac{\log|d_{L(\zeta_q)}|}{\varphi_K(q)n_{L}}
\le \frac{\log|d_L|}{n_{L}}+\log q,
\end{equation*}
and therefore,
\begin{equation}\label{replace theta error}
\begin{split}
\left|\theta(x;\delta_C\otimes\chi)-
\widetilde{\psi}(x;L(\zeta_q)/K,\delta_C\otimes\chi)\right|
&\le n_K\left(\sqrt x\log x+\frac{2}{n_{L}}\log|d_L|+2\log q\right).
\end{split}
\end{equation}

We recall that the Galois group $\Gal(K(\zeta_q)/K)$ is isomorphic to some subgroup 
of $(\zz/q\zz)^*$, which we denote by $G_{K,q}$.
Now, let $G_{K,q}^{\perp}$ denote the subgroup of Dirichlet characters modulo $q$ 
that are trivial on $G_{K,q}$.  Then two Dirichlet characters $\chi_1,\chi_2$ 
define the same map on $G_{K,q}$ 
whenever $\chi_1\equiv\chi_2\pmod{G_{K,q}^\perp}$.
Throughout, we will reserve the notation $\chi_0$ for the trivial 
Dirichlet character modulo $q$.  That is, $\chi_0(a)=1$ if $(a,q)=1$ and 
$\chi_0(a)=0$ otherwise.  
%The trivial character modulo $q=1$, we will also refer to as the principal Dirichlet character.
\begin{lma}\label{trivially twisted theta}
There exists a positive constant $c_4(L)$ 
so that for any Dirichlet character modulo $q\le x$ with 
$\chi\equiv\chi_0\pmod{G_{K,q}^\perp}$, we have
\begin{equation*}
\theta(x;\delta_C\otimes\chi)-\frac{|C|}{|G|}x\ll
x\exp\left\{
-c_4(L)\sqrt{\log x}
\right\},
\end{equation*}
where the implied constant depends only on $L$.
%If $\chi$ is primitive modulo $q$, then the $\log q$ may be omitted.
\end{lma}
%\begin{rmk}
%There are only finitely many primitive characters $\chi$ such that 
%$\chi\equiv\chi_0\pmod{G_{K,q}^\perp}$.  Thus, if we restrict attention to 
%primitive Dirichlet characters, we may replace the $\log q$ by a constant 
%depending on $K$.
%\end{rmk}
\begin{proof}
Properties of $L$-functions imply that
\begin{equation*}
\widetilde\psi(x;L(\zeta_q)/K,\delta_C\otimes\chi)
=\widetilde\psi(x;L(\zeta_q)/K,\delta_C\otimes\chi_0)
=\widetilde\psi(x;L/K,\delta_C).
\end{equation*}
By the Effective Chebotar\"ev Density Theorem~\cite[p. 458]{LO:1977},
we obtain 
\begin{equation}\label{effective chebotarev}
\widetilde\psi(x;L/K,\delta_C)=\frac{|C|}{|G|}x -\frac{|C|}{|G|}\frac{x^{\beta_0}}{\beta_0}
+O\left(x\exp\left\{-c_5n_L^{-1/2}\sqrt{\log x}\right\}\right),
\end{equation}
where $c_5$ is an absolute positive constant, 
$\beta_0$ is a potential ``exceptional zero"
of the Dedekind zeta function for $L$, and the
convention is that the term involving $\beta_0$ 
should be removed if $\beta_0$ does not exist.  
Furthermore, the constant implied by the big-$O$ is absolute.
See~\cite[pp.~413-414]{LO:1977} for the precise definition of $\beta_0$.  
Stark's bound gives
\begin{equation*}
\beta_0<\max\{1-1/(4\log|d_L|), 1-c_6/|d_L|^{1/n_L}\}.
\end{equation*}
See~\cite[p.~249]{MM:1987} or~\cite[p.~148]{Sta:1974} for example.
Therefore,
\begin{equation}\label{cheb plus stark}
\widetilde\psi(x;L/K,\delta_C)=\frac{|C|}{|G|}x
+O\left(x\exp\left\{-c_7n_L^{-1/2}\sqrt{\log x}\right\}\right),
\end{equation}
where the implied constant depends on $L$.
The result now follows by~\eqref{replace theta error}.
\end{proof}

We now turn those characters $\chi\not\equiv\chi_0\pmod{G_{K,q}^\perp}$.
For $q$ smaller than a power of $\log x$, we apply 
Proposition~\ref{SW gen} and the technique of Deuring~\cite{Deu:1935} and 
Macluer~\cite{Mac:1968} to obtain good bounds for 
$\theta(x;\delta_C\otimes\chi)$.

\begin{lma}\label{nontrivially twisted theta}
Let $M>0$ and $q\le(\log x)^M$.  There exists a positive constant 
$c_8(L,M)$ so that if $\chi\not\equiv\chi_0\pmod{G_{K,q}^\perp}$, then
\begin{equation*}
\theta(x;\delta_C\otimes\chi)\ll
x\exp\left\{
-c_8(L,M)\sqrt{\log x}
\right\},
\end{equation*}
where the implied constant depends only on $L$.
\end{lma}
\begin{proof}
Ideally, one would like to begin by decomposing 
$\delta_C$ as a linear combination of 
irreducible characters $\eta$ of $G$ and then obtain good bounds for
$|\widetilde\psi(x;L(\zeta_q)/K,\eta\otimes\chi)|$.  
However, this is complicated by the fact that $\eta$ and hence 
$\eta\otimes\chi$ may be associated to a representation 
$\rho$ of $\Gal(L(\zeta_q)/K)$ of dimension greater than one.  
In fact, unless the extension 
$L/K$ is Abelian, there will always be at least one such $\eta$.  See~\cite[p. 25]{Ser:1977} for example.

We instead employ the argument of 
Deuring~\cite{Deu:1935} and MacCluer~\cite{Mac:1968} 
to reduce to the case of one dimensional characters.  
See also~\cite[pp.~246, 250]{MM:1987} and~\cite[pp.~429-430]{LO:1977}.
In particular, we choose some Abelian subgroup $H$ of $G$ so that 
$H\cap C\neq \{\}$.  Note that if $g_C\in C$, then $H=\langle g_C\rangle$ is one possible choice.
Now, take $h_C\in H\cap C$, and let $C_H$ denote the conjugacy class of $h_C$ in $H$.
Then
\begin{equation*}
\delta_C=\frac{|C|\cdot |H|}{|G|\cdot |C_H|}
\mathrm{Ind}_{H}^{G}\delta_{C_H}.
\end{equation*}
Put $\mathcal H_q:= H\times G_{K,q}$.
By Mackey's Induction Theorem~\cite[p.~57]{Ser:1977}, we have
\begin{equation*}
\delta_C\otimes\chi=\frac{|C|\cdot |H|}{|G|\cdot |C_H|}
\mathrm{Ind}_{\mathcal{H}_q}^{\mathcal{G}_q}(\delta_{C_H}\otimes\chi),
\end{equation*}
where we are using the same symbol $\chi$ to represent the character 
on $\mathcal G_{q}$ and the associated character on $\mathcal H_q$.
Now, let $E$ be the subfield of $L(\zeta_q)$ fixed by $\mathcal H_q$,
and note that $E$ does not depend on $q$
since it is also the subfield of $L$ fixed by $H$.
By the invariance of $L$-functions under 
induction, we have
\begin{equation*}
\begin{split}
\widetilde\psi(x;L(\zeta_q)/K,\delta_C\otimes\chi)
&=\frac{|C|\cdot |H|}{|G|\cdot |C_H|}
\widetilde\psi(x;L(\zeta_q)/K,\mathrm{Ind}_{\mathcal H_q}^{\mathcal G_{q}}(\delta_{C_H}\otimes\chi))\\
&=\frac{|C|\cdot |H|}{|G|\cdot |C_H|}
\widetilde\psi(x;L(\zeta_q)/E,\delta_{C_H}\otimes\chi).
\end{split}
\end{equation*}
Now we may write
\begin{equation*}
\delta_{C_H}=\frac{|C_H|}{|H|}\sum_{\omega}\overline\omega(h_C)\omega,
\end{equation*}
where $\omega$ ranges over all the irreducible (one-dimensional) characters of $H$.
Thus,
\begin{equation}\label{bound chebychev by one-dim chebychev}
|\widetilde\psi(x;L(\zeta_q)/K,\delta_C\otimes\chi)|
\le
\frac{|C|}{|G|}
\sum_\omega|\widetilde\psi(x;L(\zeta_q)/E,\omega\otimes\chi)|.
%&\le\frac{|C|\cdot |H|}{|G|}\max_{\omega}|\widetilde\psi(x;L(\zeta_q)/E,\omega\otimes\chi)|.
\end{equation}
Since $\chi\not\equiv\chi_0\pmod{G_{K,q}^\perp}$, it follows that 
$\omega\otimes\chi$ is not trivial.  
By~\eqref{bound conductors} and the conductor-discriminant formula,
\begin{equation*}
\sum_\omega\log A(\omega\otimes\chi)\le\log\left(|d_L|q^{n_L}\right).
\end{equation*}
The result now follows by applying Proposition~\ref{SW gen} 
to~\eqref{bound chebychev by one-dim chebychev} and 
using~\eqref{replace theta error}.
\end{proof}

%Observing that $|G_{K,q}^\perp|=\varphi(q)/\varphi_K(q)$ and 
%applying Lemmas~\ref{trivially twisted theta} 
%and~\ref{nontrivially twisted theta} 
%to~\eqref{usual decomp},
%we obtain the following.
%
%\begin{prop}\label{Cheb with q dependence}
%Let $M>0$ and $q\le(\log x)^M$.
%Then
%\begin{equation*}
%\theta(x;C,q,a)
%=\frac{1}{\varphi_K(q)}\frac{|C|}{|G|}x
%+O\left(
%x\exp\left\{-c_7(L,M)\sqrt{\log x}\right\}
%\right),
%\end{equation*}
%where the implied constant depends on $L$.
%\end{prop}
%\begin{rmk}
%We, of course, could have applied the Effective Chebotar\"ev Theorem to 
%estimate $\theta(x;C,q,a)$ directly, and the result would have had the 
%same shape. However the constant appearing in the exponential would have 
%depended on $n_{L(\zeta_q)}$ which grows undesirably for our purposes in 
%Section~\ref{main proof2}.
%\end{rmk}

\section{Proof of Theorem~\ref{main thm}}

Given two Dirichlet characters, say $\chi_1$ and $\chi_2$, it follows that
$\theta(x;\delta_C\otimes\chi_1)=\theta(x;\delta_C\otimes\chi_2)$
if $\chi_1\equiv\chi_2\pmod{G_{K,q}^\perp}$.
Therefore, we set
\begin{equation*}
E(x;\delta_C\otimes\chi):=\begin{cases}
\theta(x;\delta_C\otimes\chi)-\frac{|C|}{|G|}x& 
\text{if }\chi\equiv\chi_0\pmod{G_{K,q}^\perp},\\
\theta(x;\delta_C\otimes\chi)&\text{otherwise}.
\end{cases}
\end{equation*}
For each $q$, 
let $\widehat G_{K,q}$ denote a complete set of coset 
representatives for the quotient group of Dirichlet characters modulo $G_{K,q}^\perp$.
We begin with the decomposition
\begin{equation*}
\theta(x;C,q,a)-\frac{|C|}{\varphi_K(q)|G|}x
=\frac{1}{\varphi_K(q)}
\sum_{\chi\in\widehat{G_{K,q}}}\overline\chi(a)E(x;\delta_C\otimes\chi).
\end{equation*}
Then squaring, summing over $a\in G_{K,q}$, and applying orthogonality relations as 
in~\cite[p. 170]{Dav:1980} or~\cite[p.~2740]{Smi:2009},
we have
\begin{equation*}
\begin{split}
\sum_{a\in G_{K,q}}\left|\theta(x;C,q,a)
  -\frac{|C|}{\varphi_K(q)|G|}x\right|^2
&=\frac{1}{\varphi_K(q)^2}\sum_{a\in G_{K,q}}
\left|
\sum_{\chi\in\widehat{G_{K,q}}}\overline\chi(a)E(x;\delta_C\otimes\chi)
\right|^2\\
&=\frac{1}{\varphi_K(q)}\sum_{\chi\in\widehat{G_{K,q}}}
\left|E(x;\delta_C\otimes\chi)\right|^2\\
&=\frac{1}{\varphi(q)}\sum_{\chi\bmod q}|E(x;\delta_C\otimes\chi)|^2.
\end{split}
\end{equation*}
In the last line above, the sum is taken over all Dirichlet characters 
modulo $q$.  

For each Dirichlet character $\chi$, we let $\chi_*$ denote the primitive 
character that induces $\chi$.
Summing the above over only those $q$ such that $L\cap K(\zeta_q)=K$ 
and replacing each character $\chi$ by the primitive character that 
induces it, we find that
\begin{equation*}
\sum_{q\le Q}\strut^\prime\sum_{a\in G_{K,q}}
\left(\theta(x;C,q,a)-\frac{1}{\varphi_K(q)}\frac{|C|}{|G|}x\right)^2
\ll\sum_{q\le Q}\strut^\prime(\log q)^2+
\sum_{q\le Q}\strut^\prime\frac{1}{\varphi(q)}\sum_{\chi\bmod q}|
E(x;\delta_C\otimes\chi_*)|^2.
\end{equation*}
Since $\sum_{q\le Q}(\log q)^2\le Q(\log Q)^2<xQ\log x$, 
we concentrate on the second sum on the right.
Arguing as in~\cite[p.~170]{Dav:1980}, 
we note that if $\chi$ is primitive modulo $q$, 
then $\chi$ only induces characters to moduli which are multiples of $q$.  
Hence,
\begin{equation*}
\sum_{q\le Q}\strut^\prime\frac{1}{\varphi(q)}\sum_{\chi\in\mathcal{X}(q)}
|E(x;\delta_C\otimes\chi_*)|^2
=\sum_{q\le Q}\strut^\prime\sum_{\chi}\strut^*
|E(x;\delta_C\otimes\chi)|^2\sum_{k\le Q/q}\frac{1}{\varphi(kq)},
\end{equation*}
where the $\strut^*$ on the sum over $\chi$ denotes the fact that the sum is to be 
restricted to those $\chi$ which are primitive modulo $q$.
Since $\sum_{k\le Q/q}\frac{1}{\varphi(kq)}\ll\varphi(q)^{-1}\log (2Q/q)$, 
\begin{equation*}
\sum_{q\le Q}\strut^\prime\frac{1}{\varphi(q)}\sum_{\chi\in\mathcal{X}(q)}
|E(x;\delta_C\otimes\chi_*)|^2
\ll\sum_{q\le Q}\strut^\prime\frac{1}{\varphi(q)}\log\left(2Q/q\right)
\sum_{\chi}\strut^*|E(x;\delta_C\otimes\chi)|^2;
\end{equation*}
and we see that Theorem~\ref{main thm} follows from the following proposition.
\begin{prop}\label{main prop}
Let $M>0$.
If $x(\log x)^{-M}\le Q\le x$, then
\begin{equation*}
\sum_{q\le Q}\strut'\frac{1}{\varphi(q)}\log\left(2Q/q\right)
\sum_{\chi}\strut^*|E(x;\delta_C\otimes\chi)|^2
\ll xQ\log x,
\end{equation*}
where the $'$ on the outer sum indicates that the sum is to be restricted to those $q$ such that 
$L\cap K(\zeta_q)=K$ and the $\strut^*$ on the inner sum indicates that the sum is to be 
restricted to those characters $\chi$ which are primitive modulo $q$.
The implied constant depends on $L$ and $K$.
\end{prop}

In order to prove Proposition~\ref{main prop}, 
we divide the outer sum over $q$ into 
two groupings depending on the size of $q$.
The ``large" $q$ are handled via the large sieve while the ``small" $q$ 
are handled via the results in Section~\ref{sw sect}.
For the convenience of the reader, we give a statement of the large sieve inequality as 
found in~\cite[p.~179]{IK:2004}.
\begin{thm}[Large sieve inequality]
For any complex numbers $a_n$ with $N_0<n\le N_0+N$, where $N$ is a positive integer, we have
\begin{equation*}
\sum_{q\le Q}\frac{q}{\varphi(q)}\sum_{\chi}\strut^*
\left|
\sum_{N_0<n\le N_0+N}\chi(n)a_n
\right|^2\le
\left(Q^2+N-1\right)\sum_{N_0<n\le N_0+N}|a_n|^2.
\end{equation*}
\end{thm}

\begin{proof}[Proof of Proposition~\ref{main prop}]
Fix $M>0$, and put $Q_1:=(\log x)^{M+1}$.
Applying the large sieve, we obtain
\begin{equation*}
\sum_{q\le Q}\strut^\prime\frac{q}{\varphi(q)}
\sum_{\chi}\strut^*|\theta(x;\delta_C\otimes\chi)|^2
\le n_K(Q^2+x)\sum_{\substack{\N\p\le x\\ \legendre{L/K}{\p}=C}}\left(\log\N\p\right)^2
\ll(Q^2+x)x\log x,
\end{equation*}
where the implied constant depends on $L$.
Thus, we have
\begin{equation*}
\sum_{U<q\le 2U}\strut^\prime\frac{1}{\varphi(q)}
\sum_{\chi}\strut^*|\theta(x;\delta_C\otimes\chi)|^2
\ll x\log x\left(U+xU^{-1}\right)\log\left(2Q/U\right).
\end{equation*}
Summing over $U=Q2^{-k}$, for $Q$ as specified in 
the statement of Proposition~\ref{main prop}, we have
\begin{equation*}
\begin{split}
\sum_{Q_1<q\le Q}\strut^\prime\frac{1}{\varphi(q)}\log\left(2Q/q\right)
\sum_{\chi}\strut^*|\theta(x;\delta_C\otimes\chi)|^2
&\ll x^2Q_1^{-1}(\log x)^2+xQ\log x\\
&\ll xQ\log x,
\end{split}
\end{equation*}
where the implied constant depends on $L$.
We note that 
$E(x;\delta_C\otimes\chi)\ll\theta(x;\delta_C\otimes\chi)$
by Lemma~\ref{trivially twisted theta} and Lemma~\ref{nontrivially twisted theta}.
%In fact, for $\chi$ primitive and $q$ large enough, 
%$E(x;\delta_C\otimes\chi)=\theta(x;\delta_C\otimes\chi)$.
Hence,
\begin{equation*}
\sum_{Q_1<q\le Q}\strut^\prime\frac{1}{\varphi(q)}\log\left(2Q/q\right)
\sum_{\chi}\strut^*|E(x;\delta_C\otimes\chi)|^2\ll xQ\log x,
\end{equation*}
where the implied constant depends on $L$.
This handles the large values of $q$.

For the contribution arising from the small values of $q$, viz., 
$q\le Q_1= (\log x)^{M+1}$, we apply Lemmas~\ref{trivially twisted theta} 
and~\ref{nontrivially twisted theta} to obtain the bound
\begin{equation*}
\begin{split}
\sum_{q\le Q_1}\strut^\prime\frac{1}{\varphi(q)}\log\left(2Q/q\right)
\sum_{\chi}\strut^*|E(x;\delta_C\otimes\chi)|^2
&\ll Q_1(\log Q)x^2
\exp\left\{
-c(L,M)\sqrt{\log x}
\right\}\\
&\ll x^2(\log x)^{-M}\ll xQ\log x,
\end{split}
\end{equation*}
where the implied constant depends on $L$ and $M$.
\end{proof}

\bibliographystyle{plain}
\bibliography{references}

\def\cprime{$'$}
\begin{thebibliography}{10}

\bibitem{Bar:1964}
M.B. Barban.
\newblock On the distribution of primes in arithmetic progressions ``on
  average".
\newblock {\em Dokl. Akad. Nauk UzSSR}, 5:5--7, 1964.
\newblock (Russian).

\bibitem{DH:1966}
H.~Davenport and H.~Halberstam.
\newblock Primes in arithmetic progressions.
\newblock {\em Michigan Math. J.}, 13:485--489, 1966.

\bibitem{DH:1968}
H.~Davenport and H.~Halberstam.
\newblock Corrigendum: ``{P}rimes in arithmetic progression''.
\newblock {\em Michigan Math. J.}, 15:505, 1968.

\bibitem{Dav:1980}
Harold Davenport.
\newblock {\em Multiplicative Number Theory}.
\newblock Springer-Verlag, New York, 1980.

\bibitem{DP:2004}
Chantal David and Francesco Pappalardi.
\newblock Average {F}robenius distribution for inerts in {$\Bbb Q(i)$}.
\newblock {\em J. Ramanujan Math. Soc.}, 19(3):181--201, 2004.

\bibitem{Deu:1935}
Max Deuring.
\newblock \"{U}ber den {T}schebotareffschen {D}ichtigkeitssatz.
\newblock {\em Math. Ann.}, 110(1):414--415, 1935.

\bibitem{FJKP}
Bryan Faulkner, Kevin James, Matthew King, and David Penniston.
\newblock Average {F}robenius distributions for elliptic curves over {A}belian
  extensions.
\newblock {\em Acta Arith.}
\newblock (to appear).

\bibitem{Hoo:1975}
Christopher Hooley.
\newblock On the {B}arban-{D}avenport-{H}alberstam theorem. {I}.
\newblock {\em J. Reine Angew. Math.}, 274/275:206--223, 1975.
\newblock Collection of articles dedicated to Helmut Hasse on his seventy-fifth
  birthday, III.

\bibitem{IK:2004}
Henryk Iwaniec and Emmanuel Kowalski.
\newblock {\em Analytic Number Theory}, volume~53 of {\em American Mathematical
  Society Colloquium Publications}.
\newblock American Mathematical Society, Providence, RI, 2004.

\bibitem{LO:1977}
J.~C. Lagarias and A.~M. Odlyzko.
\newblock Effective versions of the {C}hebotarev density theorem.
\newblock In {\em Algebraic number fields: $L$-functions and Galois properties
  (Proc. Sympos., Univ. Durham, Durham, 1975)}, pages 409--464. Academic Press,
  London, 1977.

\bibitem{Lan:1994}
Serge Lang.
\newblock {\em Algebraic Number Theory}, volume 110 of {\em Graduate Texts in
  Mathematics}.
\newblock Springer-Verlag, New York, second edition, 1994.

\bibitem{Lan:2002}
Serge Lang.
\newblock {\em Algebra}.
\newblock Springer-Verlag, New York, 3 edition, 2002.

\bibitem{Mac:1968}
C.~R. MacCluer.
\newblock A reduction of the \v {C}ebotarev density theorem to the cyclic case.
\newblock {\em Acta Arith.}, 15:45--47, 1968.

\bibitem{Mar:1975}
J.~Martinet.
\newblock Character theory and {A}rtin {$L$}-functions.
\newblock In {\em Algebraic number fields: {$L$}-functions and {G}alois
  properties ({P}roc. {S}ympos., {U}niv. {D}urham, {D}urham, 1975)}, pages
  1--87. Academic Press, London, 1977.

\bibitem{MM:1987}
M.~Ram Murty and V.~Kumar Murty.
\newblock A variant of the {B}ombieri-{V}inogradov theorem.
\newblock In {\em Number theory ({M}ontreal, {Q}ue., 1985)}, volume~7 of {\em
  CMS Conf. Proc.}, pages 243--272. Amer. Math. Soc., Providence, RI, 1987.

\bibitem{MM:1997}
M.~Ram Murty and V.~Kumar Murty.
\newblock {\em Non-vanishing of {$L$}-functions and Applications}, volume 157
  of {\em Progress in Mathematics}.
\newblock Birkh\"auser Verlag, Basel, 1997.

\bibitem{MP-pp}
M.~Ram Murty and Kathleen~L. Petersen.
\newblock A {B}ombieri-{V}inogradov theorem for all number fields.
\newblock (preprint).

\bibitem{Ser:1977}
Jean-Pierre Serre.
\newblock {\em Linear Representations of Finite Groups}.
\newblock Springer-Verlag, New York, 1977.
\newblock Translated from the second French edition by Leonard L. Scott,
  Graduate Texts in Mathematics, Vol. 42.

\bibitem{Ser:1981}
Jean-Pierre Serre.
\newblock Quelques applications du th\'eor\`eme de densit\'e de {C}hebotarev.
\newblock {\em Inst. Hautes \'Etudes Sci. Publ. Math.}, (54):323--401, 1981.

\bibitem{Smi:2009}
Ethan Smith.
\newblock A generalization of the {B}arban-{D}avenport-{H}alberstam {T}heorem
  to number fields.
\newblock {\em J. Number Theory}, 129(11):2735--2742, 2009.

\bibitem{Smi:2010}
Ethan Smith.
\newblock A {B}arban-{D}avenport-{H}alberstam asymptotic for number fields.
\newblock {\em Proc. Amer. Math. Soc.}, 138(7):2301--2309, 2010.

\bibitem{Sta:1974}
H.~M. Stark.
\newblock Some effective cases of the {B}rauer-{S}iegel theorem.
\newblock {\em Invent. Math.}, 23:135--152, 1974.

\bibitem{Was:1997}
Lawrence~C. Washington.
\newblock {\em Introduction to Cyclotomic Fields}, volume~83 of {\em Graduate
  Texts in Mathematics}.
\newblock Springer-Verlag, New York, second edition, 1997.

\end{thebibliography}

\end{document}